\documentclass{amsart} 

\newtheorem{theorem}{Theorem}[section]
\newtheorem{proposition}{Proposition}[section]
\newtheorem{lemma}{Lemma}[section]
\newtheorem{corollary}{Corollary}[section]
\theoremstyle{definition}

\newtheorem{example}{Example}[section]

\newtheorem{algorithm}{Algorithm}[section]

\newtheorem{remark}{Remark}[section]

\usepackage{amscd}
\usepackage{amsthm}
\usepackage{amsfonts}
\usepackage{amsbsy}
\usepackage{amsmath}
\usepackage{amssymb}
\usepackage{color}	
\usepackage{enumerate}
\usepackage[mathscr]{eucal}
\usepackage{fancyhdr}
\usepackage{mathrsfs}
\usepackage{graphicx}
\usepackage{hyperref}
\hypersetup{colorlinks=false, urlcolor=black}
\usepackage{latexsym}
\usepackage{listings}
\usepackage{longtable}
\usepackage{psfrag}
\usepackage{setspace}
\usepackage{tikz}
\usepackage{url}
\usepackage{undertilde}
\usepackage{verbatim}
\usepackage{graphics,epstopdf}
\usepackage[all]{xy}
\newcommand{\Z}{\mathbb{Z}}
\newcommand{\Q}{\mathbb{Q}}

\newcommand{\R}{\mathbb{R}}
\newcommand{\K}{\mathbb{K}}
\newcommand{\C}{\mathbb{C}}
\newcommand{\F}{\mathbb{F}}
\newcommand{\EE}{\mathbb{E}}
\newcommand{\LL}{\mathbb{L}}

\begin{document}

\title[Arithmetic matrices for number fields I]{Arithmetic matrices for number fields I}

\author[Samuel A. Hambleton]{Samuel A. Hambleton}

\address{School of Mathematics and Physics, The University of Queensland, St. Lucia, Queensland, Australia 4072}

\email{sah@maths.uq.edu.au}


\date{September 25, 2018.}


\begin{abstract}
We provide a simple way to add, multiply, invert, and take traces and norms of algebraic integers of a number field using integral matrices. With formulas for the integral bases of the ring of integers of at least a significant proportion of numbers fields, we obtain explicit formulas for these matrices and discuss their generalization. These results are useful in proving statements about the particular number fields they work in. We give a meaningful diagonalization that is helpful in this regard and can be used to define such matrices in general. The matrix identities provided suggest that consideration of numerical methods in linear algebra may have applications in algebraic number theory. Multiplication of several algebraic integers at once might be more efficiently implemented using methods of efficient matrix multiplication. If we are multiplying an algebraic integer in a field of degree $n$ by each of $n$ algebraic integers, then this can be done in $O \left( n^{\log_2(7)} \right)$ multiplications. The matrix identities given here generalize Brahmagupta's identity upon taking determinants, which are multiplicative.
\end{abstract}

\maketitle

\section{Introduction}\label{intro}

It is well know that number fields can be generated by a single element $\zeta $. This fact allows efficient multiplication of the elements of a field $\Q (\zeta )$ of large degree $n$ over $\Q$, as univariate polynomial multiplication may be performed using the fast Fourier transform (FFT) \cite{CT2}. However, the ring of integers of a number field does not always possess a power basis. When we want to express the product of two algebraic integers in terms of the same integral basis in which they were defined, use of the field generator $\zeta $ is perhaps not always the best approach for achieving this. Moreover, when the product of several algebraic integers is required, such as may be the case when we wish to obtain the integral basis of the product of two ideals, matrices provide a means of doing this efficiently. The aim of this article is to introduce a technique for constructing integer matrices that can be used to more conveniently and possibly more efficiently perform arithmetic in the ring of integers of a number field. Algorithms for efficient implementation of matrix multiplication include Strassen's algorithm \cite{Strassen2} and the Winograd-Waksman algorithm \cite{Waksman2,Winograd2}. However, irrespective of efficiency, use of the integral matrices that will be introduced here can greatly simplify proofs of statements about the fields they work in since the diagonalized form of the matrix can be used. There are numerous examples of this in the case that $n = 3$ found in \cite{cfg2}. 

The main result of this article is a generalization of Proposition \ref{mainres} below, Proposition \ref{genclaim}, so that it may describe more number fields. These results show that algebraic integers $\alpha = \sum_{j = 0}^{n-1} x_j \rho_j$, can be encoded in the $n$ by $n$ matrix $M_{\EE}^{(\alpha )} = \left[ a_{ij} \right]$ given by the formulas for the entries $a_{ij}$, (and later generalized in Section \ref{essent})
\begin{eqnarray}
\label{defaijone} a_{11} & = & x_0 , \\
\label{defaijtwo} a_{i1} & = & x_{i-1} , \ \text{for} \ i > 1 , \\
\label{defaijthreee} a_{1j} & = & - a_{n+1} \sum_{k = 1}^{j-1} a_k x_{k+n-j}, \ \text{for} \ j > 1 , \\
\label{defaijfour} a_{ij} & = &  \sum_{k = 1}^{j-1} a_k x_{k+i-j-1} \ \text{for} \ i > j > 1 , \\
\label{defaijfive} a_{ij} & = & \delta_{ij} x_0 - \sum_{k = j}^{m_{ij}} a_k x_{k+i-j-1} \ \text{for} \ j \geq i > 1 , \\
\nonumber m_{ij} & = & \min (n - i + j, n + 1) .
\end{eqnarray}  
$\delta_{ij}$ is the Kronecker-delta symbol, ($1$ if $i=j$ and $0$ otherwise). When $n = 3$, the formulas produce the matrix 
\begin{equation*}
N_{\K }^{(\alpha )} = \left(
\begin{array}{ccc}
 u & -a d y & -a d x-b d y \\
 x & u-b x-c y & -c x-d y \\
 y & a x & u-c y \\
\end{array}
\right) ,
\end{equation*}
where we rename letters $x_0 = u$, $x_1 = x$, $x_2 = y$, $a_1 = a$, $a_2 = b$, $a_3 = c$, and $a_4 = d$. Moreover, the proposition shows that the matrix $M_{\EE}^{(\alpha )}$ has a diagonalized form which gives it several properties that are useful in working with number fields. 

\begin{proposition}\label{mainres}
Let $\EE = \Q (\zeta )$ be a number field of degree $n$ over $\Q$, where $\zeta $ is a root of the irreducible polynomial $f(x) = a_1 x^n + a_2 x^{n-1} + \dots + a_{n+1} \in \Z [x]$ with roots $\zeta_0, \zeta_1, \dots , \zeta_{n-1}$; $\zeta_0 = \zeta $. Let $\left\{ \rho_0, \rho_1, \dots , \rho_{n-1} \right\}$ be an integral basis for the ring of integers $\mathcal{O}_{\EE}$ of $\EE$ satisfying $ \rho_j = \sum_{k = 1}^{j} a_k \zeta^{j+1-k} \ (j \geq 1) , \ \rho_0 = 1$. Let $\kappa_{0}, \kappa_1, \dots , \kappa_{n-1}$ be the embeddings of $\EE$ in $\C$ so that $\kappa_{t}\left( \zeta_{u} \right) = \zeta_{v}$, $v \equiv t + u \pmod{n}$. Let $\Gamma_{\EE} = \left[ \kappa_{i-1} \left( \rho_{j-1} \right) \right]$ and let $\Theta_{\EE}^{(\alpha )} = \left[ \delta_{ij} \kappa_{i-1} \left( \alpha \right) \right] $ so that $\Theta_{\EE }^{(\alpha )}$ is a diagonal matrix. Let $N_{\EE}^{(\alpha )} = \Gamma_{\EE}^{-1} \Theta_{\EE}^{(\alpha )} \Gamma_{\EE}$, and let $M_{\EE}^{(\alpha )} = \left[ a_{ij} \right]$ be the matrix defined by \eqref{defaijone} to \eqref{defaijfive}. Then the following properties of the matrix $N_{\EE}^{(\alpha )}$ hold for $\alpha , \beta \in \mathcal{R} \left[ \rho_0, \rho_1, \dots , \rho_{n-1} \right]$, where $\mathcal{R}$ is a commutative ring with $x_0, x_1, \dots , x_{n-1} \in \mathcal{R}$, $\alpha = \sum_{j=0}^{n-1} x_j \rho_j$. 
\begin{enumerate}
\item $N_{\EE}^{(\alpha )} = M_{\EE}^{(\alpha )}$.
\item $N_{\EE}^{(\alpha )} + N_{\EE}^{(\beta )} = N_{\EE}^{(\alpha + \beta )}$.
\item $N_{\EE}^{(\alpha )} N_{\EE}^{(\beta )} = N_{\EE}^{(\beta )} N_{\EE}^{(\alpha )} = N_{\EE}^{(\alpha \beta )}$.
\item The trace of the matrix $N_{\EE}^{(\alpha )}$ is equal to the trace of $\alpha $.
\item The determinant of the matrix $N_{\EE}^{(\alpha )}$ is equal to the norm of $\alpha $.
\item $N_{\EE}^{(1 / \alpha )} = \left( N_{\EE}^{(\alpha )} \right)^{-1}$. 
\item $N_{\EE}^{(\alpha )}$ has entries in $\Z$ if and only if $\alpha \in \mathcal{O}_{\EE}$. 
\end{enumerate}
\end{proposition}

\noindent The analogue of this result for cubic fields has been helpful in proving results about ideals and lattices of cubic fields; see \cite{cfg2}. Proposition \ref{mainres} may have similar applications. We call the symbolic expression in $a_1, a_2, \dots , a_{n+1}, x_0, x_1, \dots , x_{n-1}$ for $N_{\EE}^{(\alpha )}$ the {\em arithmetic matrix} for the number field $\EE$. Hermann Weyl \cite{Weyl2} discussed matrices satisfying some items of the proposition. However, Weyl described such matrices in terms of a basis for the number field over $\Q$ rather than an integral basis of the ring of integers. We present matrices in this article that may be thought of in the context of representation theory \cite{Goodman2} as matrix representations of the ring of integers of a number field. In the context of algebraic geometry, the equation 
\begin{equation*}
\det \left( N_{\EE}^{(\alpha )} \right) = 1
\end{equation*}
is the kernel of the norm map
\begin{equation*}
N_{\EE / \Q } : \ \alpha \longmapsto N_{\EE / \Q }(\alpha ) ,
\end{equation*} 
known as a norm one torus; see for example \cite{HambLemm12,LeBruyn42,Lem0312,Lemnormtori42,LemmParam72,Voskbook42,Vosktwodimone42,Vosktwodimtwo42}.   

The sections are presented in the following order.
\begin{itemize}
\item Proof of Proposition \ref{mainres}, \S \ref{general}. 
\item Quartic fields, as an example with remarks on classical invariants \S \ref{five}.
\item Existence of binary forms of discriminant equal to a number field of the same degree, and existence of an integral basis of the prescribed form \S \ref{essent}.
\item Efficient implementation of matrix multiplication as it applies to multiplication of algebraic integers, \S \ref{matmult}. 
\end{itemize}
\noindent In a subsequent article \cite{hambpara2} we will show that the arithmetic matrices can be used to parameterize rings by binary forms. 

It appears as though there is not always an integral basis of the form given in Proposition \ref{mainres}. However, the modifications required are minor in order for the matrices we describe to be applicable to the integral bases of the rings of integers of significantly more number fields. Since it is simpler to first apprehend and prove the statement as presented in Proposition \ref{mainres}, we will retain the current presentation of the result until \S \ref{essent}. In that section we will replace \eqref{defaijone} to \eqref{defaijfive} in cases in which there is no irreducible binary form of degree $n$ having discriminant equal to the discriminant of a number field of the same degree. We state this in Proposition \ref{genclaim}. When $1 < n < 4$, we can always find such a binary form.  

\section{Number fields of degree $n$}\label{general}

Let $\EE = \Q (\zeta )$ be a number field of degree $n$ over $\Q$, where 
\begin{equation}\label{polyf}
f(x) = a_1 x^n + a_2 x^{n-1} + \dots + a_{n+1} \in \Z [x]
\end{equation}
with roots $\zeta_0, \zeta_1, \dots , \zeta_{n-1}$; $\zeta_0 = \zeta $. Let $\{ \rho_0 , \rho_1, \dots , \rho_n \}$ be an integral basis for the ring of integers $\mathcal{O}_{\EE}$, and let $\theta = \sum_{j = 0}^{n-1} x_j \rho_j $, where the $x_j$ are indeterminants. In order to find an arithmetic matrix $N_{\EE }^{(\theta )}$, let $\Theta_{\EE}^{\left( \theta \right)}$ be the $n \times n$ diagonal matrix with diagonal entries $\kappa_{i-1} (\theta )$, where the $\kappa_{i-1}$ $(i = 1,2, \dots, n)$ are the embeddings of $\F$ in $\C$, and let $\Gamma_{\EE} = \left[ a_{ij} \right]$ with entries $a_{ij} = \kappa_{i-1} \left( \rho_{j-1} \right)$, so that $\det \left( \Gamma_{\F} \right)^2 $ is the discriminant of $\EE$. We begin by defining 
\begin{equation}\label{fin}
 N_{\EE}^{\left( \theta \right) } = \Gamma_{\EE}^{-1} \Theta_{\EE}^{\left( \theta \right)} \Gamma_{\EE} ,
\end{equation}
where preferably we have a formula for the $\rho_j$ in terms of invariants of $\EE$. In the case that $\EE = \K$ is a cubic field, we can assume that the index form $\mathcal{C} = (a, b, c, d)$ that is used to define $\K$ is a reduced binary cubic form, which makes the form $(a, b, c, d)$ unique in the $\text{GL}_2(\Z )$ class it belongs to. Given that the matrix $N_{\EE}^{\left( \theta \right) }$ is defined to satisfy \eqref{fin}, it is easy to see that the determinant of $ N_{\EE}^{\left( \theta \right) }$ must coincide with the norm of $\theta $ over $\Q$, and many of the statements of Proposition \ref{mainres} hold by inspection. In fact: 
\begin{enumerate}
\item The arithmetic matrices are additive, $N_{\EE}^{ \left( \theta_1 + \theta_2 \right) } = N_{\EE}^{ \left( \theta_1 \right) } + N_{\EE}^{ \left( \theta_2 \right) } $.
\item The arithmetic matrices are multiplicative, $N_{\EE}^{ \left( \theta_1 \theta_2 \right) } = N_{\EE}^{ \left( \theta_1 \right) } N_{\EE}^{ \left( \theta_2 \right) }$.
\item The traces coincide since $\text{tr}(A B) = \text{tr}(B A)$, see Lang \cite{Lang2}.
\item The determinant coincides with the norm. 
\end{enumerate}
Notice that the $\kappa_{i - 1}(\theta )$ are the eigenvalues of the matrix $N_{\EE }^{(\theta )}$. These results suggest that many of the algorithms of linear algebra may have applications in algebraic number theory. Some algorithms in numerical analysis applied to number theory will have limitations however. If we use the FFT to multiply two algebraic integers $\alpha , \gamma $ belonging to $\mathcal{O}_{\EE } \ = \left\{ \rho_0, \rho_1, \dots , \rho_{n-1} \right\}$, where $\EE = \Q (\zeta )$ is of degree $n$ over $\Q$, we would need to express $\alpha , \gamma $ as 
\begin{align*}
\alpha & = \sum _{j = 0}^{n-1} a_j \zeta^j , & \gamma & = \sum _{j = 0}^{n-1} c_j \zeta^j , & & a_j , c_j \in \Q ,
\end{align*}
using a linear relationship between the algebraic integers $\rho_0, \rho_1, \dots , \rho_{n-1} $ and the basis $1, \zeta , \zeta^2 , \dots , \zeta^{n-1}$ of $\EE$. However, we note that translating between the basis of $\EE $ and the basis of $\mathcal{O}_{\EE }$ is not without multiplications. In fact, we have
\begin{align}\label{gammadecomp}
 \Gamma_{\EE } & = \Xi A, 
\end{align} 
where $\zeta_0 = \zeta$, $\zeta_i$ are the roots of $f(x)$, and $\Xi = \left[ \zeta_{i-1}^{j-1} \right] $,
\begin{align*}
\Xi & = \left(
\begin{array}{ccccc}
 1 & \zeta &   \dots & \zeta^{n-1} \\
 1 & \zeta_{1}  & \dots & \zeta_{1}^{n-1} \\
 1 & \zeta_{2}  & \dots & \zeta_{2}^{n-1} \\
 1 & \zeta_{3}  & \dots & \zeta_{3}^{n-1} \\
 \vdots & \vdots  & \ddots & \vdots \\
 1 & \zeta_{n-1} & \dots &  \zeta_{n-1}^{n-1} \\
\end{array}
\right) , & A & = \left(
\begin{array}{cccccc}
 1 & 0 & 0 & 0 & \dots & 0 \\
 0 & a_1 & a_2 & a_3 & \dots & a_{n-1} \\
 0 & 0 & a_1 & a_2 & \dots & a_{n-2} \\
 0 & 0 & 0 & a_1 & \dots & a_{n-3} \\
 \vdots & \vdots & \vdots & \vdots & \ddots & \vdots \\
 0 & 0 & 0 & 0 & \dots & a_1 \\
\end{array}
\right) .
\end{align*}

We have seen that the matrices $N_{\EE }^{ \left( \theta \right) }$ have a diagonalization given by their definition. The diagonal elements of $\Theta_{\EE }^{ \left( \theta \right) }$ are the eigenvalues of $N_{\EE }^{ \left( \theta \right) }$ and the columns of $\Gamma_{\EE}^{-1}$ are the eigenvectors of $N_{\EE}^{ \left( \theta \right) }$. 

Frobenius \cite{Frobenius2}, Drazin \cite{Drazin2}, and others considered the eigenvalues and eigenvectors of matrices that commute with one another. It is well known \cite{Williamson2} that commuting matrices over an algebraically closed field are simultaneously triangularizable; if $M_{1} , M_{2} , \dots , M_{m}$ are $n \times n$ matrices which commute with each other, then there exists a matrix $P$ such that $P^{-1} M_{j} P$ is upper triangular for all $j \in \{ 1, 2, \dots , m \}$. This prompts the question of whether we can understand results on commuting matrices in the context of algebraic number theory. It seems however that the arithmetic matrices may not be simultaneously placed in Hermite normal form (HNF) using one matrix $P$ for which a formula for the entries of $P$ is known. In many cases, for example when $\EE$ is a cubic field, the columns of the matrix $N_{\EE }^{ \left( \theta \right) }$ may be used to give an integral basis for the principal ideal $\mathfrak{a} = (\alpha )$ of $\mathcal{O}_{\EE }$. It is often convenient to put such matrices in HNF in order to compare the ideals that they describe.  

Again let $\EE = \Q (\zeta )$, where $\zeta $ is a root of the irreducible polynomial $$f(x) = \sum_{k = 1}^{n} a_k x^{n+1-k} $$ of degree $n$ with coefficients $a_k \in \Z$. If $\Delta $ is the discriminant of a field of degree $n$, and the discriminant of the polynomial $f(x)$ is equal to $\Delta$, then it is reasonable to suspect that $\mathcal{A} = \{ \rho_j \}_{j = 0}^{n-1}$ is an integral basis for the ring of integers $\mathcal{O}_{\EE }$, where $\rho_0 = 1$ and if $j > 0$, then $\rho_j = \sum_{k = 1}^{j} a_k \zeta^{j+1-k}$. When this is indeed the case, as we will see in Theorem \ref{inbasthm}, we can then write the algebraic integers $\alpha \in \mathcal{O}_{\EE }$ in the form $\alpha = \sum_{k = 0}^{n-1} x_k \rho_k $. Noting that when $n = 2$, we can define $N_{\EE }^{(\alpha )} = \left(
\begin{array}{cc}
 x_0 & -a_1 a_3 x_1 \\
 x_1 & x_0-a_2 x_1 \\
\end{array}
\right) $. The $n \times n$ arithmetic matrices $N_{\EE }^{(\alpha )}$ for $\EE$ equal to $\LL = \Q (\sqrt{D})$ (a quadratic field) or $\K = \Q (\delta )$ (a cubic field) satisfy $N_{\EE }^{(\alpha )} = \left[ a_{ij} \right]$, where
\begin{eqnarray}
\label{defaijonea4} a_{11} & = & x_0 , \\
\label{defaijtwoa4} a_{i1} & = & x_{i-1} , \ \text{for} \ i > 1 , \\
\label{defaijthreeea4} a_{1j} & = & - a_{n+1} \sum_{k = 1}^{j-1} a_k x_{k+n-j}, \ \text{for} \ j > 1 , \\
\label{defaijfoura4} a_{ij} & = &  \sum_{k = 1}^{j-1} a_k x_{k+i-j-1} \ \text{for} \ i > j > 1 , \\
\label{defaijfivea4} a_{ij} & = & \delta_{ij} x_0 - \sum_{k = j}^{m_{ij}} a_k x_{k+i-j-1} \ \text{for} \ j \geq i > 1 , \\
\nonumber m_{ij} & = & \min (n - i + j, n + 1) .
\end{eqnarray}
The Kronecker-delta symbol is denoted by $\delta_{ij}$ here defined in Proposition \ref{mainres}. 

If we now permit $n$ to exceed $3$, then remarkably the $N_{\EE }^{(\alpha )}$ defined in this way commute. For example, if we choose $n = 5$, then we can use the definition of the entries $a_{ij}$ of $N_{\EE }^{(\alpha )}$ to construct the following $5 \times 5$ arithmetic matrix for the quintic field $\EE $, where we have again replaced $a_1 = a$, $a_2 = b$, $a_3 = c$, $a_4 = d$, $a_5 = e$, $a_6 = f$, $x_0 = u$, $x_1 = x$, $x_2 = y$, $x_3 = z$, $x_4 = w$. 
\tiny 
\begin{equation*}
\left(
\begin{array}{ccccc}
 u & -a f w & -f (b w+a z) & -f (c w+a y+b z) & -f (d w+a x+b y+c z) \\
 x & u-e w-b x-c y-d z & -f w-c x-d y-e z & -d x-e y-f z & -e x-f y \\
 y & a x & u-e w-c y-d z & -f w-d y-e z & -e y-f z \\
 z & a y & a x+b y & u-e w-d z & -f w-e z \\
 w & a z & a y+b z & a x+b y+c z & u-e w \\
\end{array}
\right) .
\end{equation*}
\normalsize 
A calculation shows that $N_{\EE }^{(\alpha )} N_{\EE }^{(\beta )} - N_{\EE }^{(\beta )} N_{\EE }^{(\alpha )} = [0]$, the $5 \times 5$ matrix with zero entries. 

The following lemma is helpful in completing the proof of Proposition \ref{mainres}, as Item 7 of the statement is non-trivial.

\begin{lemma}\label{samematr}
Let $M_{\EE}^{(\alpha )}$ be defined by the equations \eqref{defaijonea4} to \eqref{defaijfivea4}, where $\EE$ is a number field of degree $n$, and let $N_{\EE}^{(\alpha )} = \Gamma_{\EE}^{-1} \Theta_{\EE }^{(\alpha )} \Gamma_{\EE }$. Then $M_{\EE}^{(\alpha )} = N_{\EE}^{(\alpha )}$.
\end{lemma}

\begin{proof}
We show that the result holds for each $\alpha_j = x_j \rho_j$, where $\rho_j$ is given by $\rho_j = \sum_{k = 1}^{j} a_k \zeta^{j+1-k} \ (j \geq 1) , \ \rho_0  = 1$, and then use the additive property of each matrix, $N_{\EE}^{\left( \alpha \right) } = \sum_{j = 0}^{n-1} N_{\EE}^{\left( \alpha_j \right) } = N_{\EE}^{\left( \sum_{j = 0}^{n-1} \alpha_j \right) }$. First, we have 
\begin{equation*}
N_{\EE}^{\left( \rho_0 \right) } = N_{\EE}^{\left( 1 \right) } = \Gamma_{\EE}^{-1} \Theta_{\EE}^{ \left( 1 \right) } \Gamma_{\EE} = \left[ \delta_{ij} \right] .
\end{equation*} 
Since the equations describing $M_{\EE}^{ \left( \rho_0 \right) }$ give the same result, we see that by multiplying by $x_0$, we have $M_{\EE}^{ \left( \alpha_0 \right) } = N_{\EE}^{\left( \alpha_0 \right) }$. 

Now we consider the $\alpha_q$ for $q = 1, 2, \dots , n-1$. First observe that
\begin{align}\label{gammadecomp2}
 \Gamma_{\EE} & = \Xi A, 
\end{align} 
where $\zeta_0 = \zeta$, $\zeta_i$ are the roots of $f(x)$, and $\Xi = \left[ \zeta_{i-1}^{j-1} \right] $, and $A$ was defined just after \eqref{gammadecomp}. We define $G_q = A M_{\EE}^{\left( \rho_q \right) } A^{-1}$ and construct a formula for $G_q$ for each $q = 1, \dots, n-1$. We must show that for $q > 1$, $\Xi G_q = \Theta^{\left( \rho_q \right) } \Xi $. Clearly
\begin{equation}\label{pmfgi}
\Theta^{\left( \rho_q \right) } \Xi = \left[ \sum_{k=1}^{q} a_k \zeta_{i-1}^{j+q-k} \right] .
\end{equation}
Now $G_q = \left( R_q \ S_q \right)$, where 
\begin{align*}
R_q & = \left(
\begin{array}{cccc}
 0 & 0 & \dots & 0 \\
 a_q & 0 & \dots & 0 \\
 a_{q-1} & a_q & \dots & 0 \\
 a_{q-2} & a_{q-1} & \dots & 0 \\
 \vdots & \vdots & \ddots & \vdots \\
 a_1 & a_2 & \ddots & 0 \\
 0 & a_1 & \ddots & a_q \\
 0 & 0 & \ddots & a_{q-1} \\
 \vdots & \vdots & \ddots & \vdots \\
 0 & 0 & \ddots & a_2 \\
 0 & 0 & \dots & a_1 \\
\end{array}
\right) , & S_q & = \left(
\begin{array}{cccc}
 -a_{n+1}     & 0        & \dots  & 0 \\
 -a_{n} & -a_{n+1}   & \dots  & 0 \\
 -a_{n-1} & -a_{n} & \dots  & 0 \\
 \vdots   & \vdots   &    & \vdots \\
 -a_q      & -a_{q+1}   & \ddots & 0 \\
 0        & -a_q      & \ddots & 0 \\
 0        & 0        & \ddots & -a_{n+1} \\
 0        & \vdots   & \ddots & -a_{n} \\
 \vdots   & \vdots   & \ddots & \vdots \\
 0        &        0 & \ddots & -a_q \\
\end{array}
\right) ,
\end{align*}
$R_q$ is an $n \times (n-q)$ matrix, and $S_q$ is an $n \times q$ matrix. If we let 
\begin{eqnarray*}
f_{q+1} (x) & = & \varepsilon x^{q+1} + a_q x^{q} + a_{q-1} x^{q-1} + \dots + a_{1} , \\
f_{n+1} (x) & = & - a_{n+1} x^{n+1} - a_n x^{n} - a_{n-1} x^{q-1} - \dots - a_{1} ,
\end{eqnarray*}
then the Sylvester matrix associated with the resultant of the polynomials $f_{q+1} (x)$ and $f_{n+1} (x)$ is equal to the transpose of the matrix $G_q$, once we replace $\varepsilon$ with $0$. It is easy to see that since $a_1 \zeta^{n} + a_2 \zeta^{n-1} + \dots + a_{n+1} = 0$, we have 
\begin{equation*}
\Xi G_q = \left[ \sum_{k=1}^{q} a_k \zeta_{i-1}^{j+q-k} \right] .
\end{equation*}
Since this coincides with \eqref{pmfgi}, we have $\Xi G_q = \Theta^{\left( \rho_q \right) } \Xi$. It follows that 
\begin{eqnarray*}
N_{\EE}^{\left( \rho_j \right) } & = & \Gamma_{\EE}^{-1} \Theta_{\EE}^{ \left( \rho_j \right) } \Gamma_{\EE} = A^{-1} \Xi^{-1} \Theta_{\EE}^{ \left( \rho_j \right) } \Xi A , \\
                                  & = & A^{-1} G_j A  = A^{-1} A M_{\EE}^{\left( \rho_j \right) } A^{-1} A , \\ 
                                  & = & M_{\EE}^{\left( \rho_j \right) } .
\end{eqnarray*} 
Multiplying by $x_j$, we have $N_{\EE}^{\left( x_j \rho_j \right) } = M_{\EE}^{\left( x_j \rho_j \right) }$. Taking sums from $j = 0$ to $j = n - 1$, we obtain $M_{\EE}^{(\alpha )} = N_{\EE}^{(\alpha )}$.
\end{proof}

The following result generalizes Belabas' integral basis of a cubic field \cite{Belabas2}. In Section \ref{essent} we will consider the existence of the binary forms satisfying the conditions of the theorem. 

\begin{theorem}\label{inbasthm}
Let 
\begin{equation}\label{eqfpolthm}
f(x) = a_1 x^n + a_2 x^{n-1} + \dots a_{n} x + a_{n+1}  
\end{equation}
be an irreducible polynomial of degree $n$ in $\Z [x]$ such that the discriminant of $f(x)$ is equal to the discriminant $\Delta $ of a number field $\EE = \Q (\zeta )$ of degree $n$ over $\Q $, where $f(\zeta ) = 0$. Then an integral basis for the ring of integers of $\EE$ is given by
\begin{align}\label{intbasn}
\mathcal{O}_{\EE} & = \left\{ \rho_0 , \rho_1 , \rho_2 , \dots , \rho_{n-1} \right\} , & \rho_j & = \sum_{k = 1}^{j} a_k \zeta^{j+1-k} \ (j \geq 1) , & \rho_0 & = 1.
\end{align}
\end{theorem}

\begin{proof}
Since \eqref{gammadecomp} holds, where $\zeta_{i}$ are the roots of $f(x)$ given by \eqref{eqfpolthm}, taking determinants of the matrices in \eqref{gammadecomp} and squaring the result gives 
\begin{equation}\label{disceqn}
\Delta = \det \left( \Gamma_{\EE} \right)^2 = a_{1}^{2 n-2} \prod_{i < j} \left( \zeta_{i} - \zeta_{j} \right)^2 . 
\end{equation}

Assume that there exist $p_0, p_1, \dots , p_{n-1} \in \Q$ such that $\sum_{j=0}^{n-1} p_j \rho_j = 0$. Then applying the embeddings $\kappa_i$ of $\EE$ in $\C$, we have 
\begin{equation*}
\Gamma_{\EE} \left(
\begin{array}{cccc}
 p_0 & p_1 & \dots & p_{n-1} \\
\end{array}
\right)^T = \left(
\begin{array}{cccc}
 0 & 0 & \dots & 0 \\
\end{array}
\right)^T . 
\end{equation*}
Since $\Gamma_{\EE}$ is invertible, we must have $p_j = 0$ for $j = 0, 1, \dots, n-1$. It follows that $\left\{ \rho_j \right\}_{j=0}^{n-1}$ is a basis for $\EE$ over $\Q$. 

It is easy to verify that $\rho_1, \rho_2, \dots \rho_{n-1}$ are respectively roots of the monic polynomials 
\begin{eqnarray*}
f_j(x) & = & \det \left( N_{\EE}^{\left( x - \rho_j \right) } \right) \ (j = 1, 2, \dots n-1) , \\
       & = & \det \left( \Gamma_{\EE}^{-1} \Theta_{\EE}^{ \left( x - \rho_j \right) } \Gamma_{\EE} \right) = \prod_{j=0}^{n-1} \left( x - \rho_j \right) .
\end{eqnarray*}
By Lemma \ref{samematr}, the coefficients of $f_j(x)$ are rational integers. It follows that $\rho_1$, $\rho_2, \dots$, $\rho_{n-1} \in \mathcal{O}_{\EE}$.

Let $\mathcal{O} $ be the set of all $\alpha = \sum_{j=0}^{n-1} x_j \rho_j $ such that for each $j = 0, 1, \dots , n - 1$, we have $x_j \in \Z$. Clearly $\mathcal{O}$ is a sub-module of $\EE$ and $\mathcal{O}$ is a subring of $\EE$. $\mathcal{O}$ contains $1 \ ( = \rho_0 ) $ and a basis of $\EE$ over $\Q$. Therefore $\mathcal{O}$ is an order of $\EE$ and, further, $\mathcal{O} \subseteq \mathcal{O}_{\EE}$. We have already shown in \eqref{disceqn} that the discriminant of $\mathcal{O}$ is equal to $\Delta$, the discriminant of the polynomial $f(x)$. Thus if $\Delta$ is the discriminant of the field $\EE$, then $\mathcal{O}$ is the maximal order of $\EE$ and we have $\mathcal{O} = \mathcal{O}_{\EE}$. 
\end{proof}

It is worth mentioning that there is slight down-side to a general formula for the integral basis of a number field. With respect to Theorem \ref{inbasthm}, one must find a generating polynomial $f(x)$ that has the polynomial discriminant equal to that of a number field $\EE = \Q (\zeta )$ of the same degree as $f$, with $f(\zeta ) = 0$. Alternatively, when one does not exist, we must find a binary form satisfying other conditions that we consider in Section \ref{essent}. Further, occasionally $\Z [\zeta ]$ is the ring of integers of $\EE$, which has a simpler description than the basis obtained in Theorem \ref{inbasthm}. For example, if $\zeta_n$ is a primitive $n$-th root of unity, then the ring of integers of the cyclotomic field $\Q \left( \zeta_n \right)$ is simply $\Z \left[ \zeta_n \right]$, see \cite[pp. 11]{Washington2}, which does not disagree with Theorem \ref{inbasthm}. 

\begin{corollary}\label{cortothm}
The entries of $N_{\EE}^{(\alpha )}$ are rational integers if and only if $\alpha \in \mathcal{O}_{\EE}$. 
\end{corollary}

\begin{proof}
Since the $\rho_j = \sum_{k = 1}^{j} a_k \zeta^{j+1-k} \ (j \geq 1) , \ \rho_0  = 1$ are all algebraic integers, if the entries of $N_{\EE}^{(\alpha )}$ are rational integers, then $\Gamma_{\EE} N_{\EE}^{(\alpha )}$ has entries in $\mathcal{O}_{\EE }$. It follows that $\Theta_{\EE}^{(\alpha )} \Gamma_{\EE} \ \left( = \Gamma_{\EE} N_{\EE}^{(\alpha )} \right)$ has entries in $\mathcal{O}_{\EE }$. Notice that $\alpha $ is the element in the first row and first column of $\Theta_{\EE}^{(\alpha )} \Gamma_{\EE}$, which must belong to $\mathcal{O}_{\EE}$. Conversely, if $\alpha \in \mathcal{O}_{\EE}$, then the coefficients $x_j$ of $\rho_j$ in the expression of $\alpha $ in terms of the integral basis \eqref{intbasn} must be rational integers. By Lemma \ref{samematr}, the entries of $N_{\EE}^{(\alpha )}$ are rational integers.
\end{proof}

Once we observe that in general the indeterminants $x_j$ may belong to a commutative ring $\mathcal{R}$ in order for the results in this section to hold, Proposition \ref{mainres} follows. 

\section{Quartic fields}\label{five}

In this section, as an example, we will consider the homogeneous irreducible polynomial $\mathcal{V}(x, y)$ over $\Q$ with coefficients $a,b,c,d,e$ in $\Z$ given by 
\begin{equation*}
\mathcal{V}(x, y) = a x^4 + b x^3 y + c x^2 y^2 + d x y^3 + e y^4 .
\end{equation*} 
For similar examples for quadratic and cubic fields, see \cite{hamblbrahm2,cfg2}. Let $\zeta $ be a root of $\mathcal{V}(x, 1)$ and assume that the field $\EE = \Q (\zeta )$ has discriminant $\Delta $, where 
\begin{eqnarray}
\label{discquad} \Delta & = & 256 a^3 e^3-192 a^2 b d e^2-128 a^2 c^2 e^2+144 a^2 c d^2 e-27 a^2 d^4 \\
\nonumber        &   & +144 a b^2 c e^2-6 a b^2 d^2 e-80 a b c^2 d e+18 a b c d^3+16 a c^4 e \\
\nonumber        &   & -4 a c^3 d^2-27 b^4 e^2+18 b^3 c d e-4 b^3 d^3-4 b^2 c^3 e+b^2 c^2 d^2 
\end{eqnarray}
is the discriminant of the binary quartic form $\mathcal{V}(x, y)$. We know that 
\begin{equation*}
\mathcal{A} = \left\{ 1 , \rho_1 , \rho_2 , \rho_3 \right\} = \left\{ 1 , a \zeta , a \zeta^2 + b \zeta , a \zeta^3 + b \zeta^2 + c \zeta \right\}
\end{equation*}
is an integral basis for the ring of integers $\mathcal{O}_{\EE}$ by Theorem \ref{inbasthm}. We can now exhibit the arithmetic matrix $N_{\EE }^{(\theta )}$, where $\alpha = u + x \rho_1 + y \rho_2 + z \rho_3 $, as 
\begin{equation}\label{arith4}
N_{\EE}^{(\alpha )} = \left(
\begin{array}{cccc}
 u & -a e z & -e (a y + b z) & -e (a x + b y + c z) \\
 x & u-b x-c y-d z & -c x - d y - e z & - d x - e y \\
 y & a x & u - c y - d z & - d y - e z \\
 z & a y & a x + b y & u - d z \\
\end{array}
\right) .
\end{equation}
We clearly see a formula for the trace of $\alpha $ is given by 
\begin{equation*}
\text{Tr}(\alpha ) = 4 u - b x - 2 c y - 3 d z .
\end{equation*}
Letting $t = \text{Tr}(\alpha )$, we have 
\begin{equation}\label{equu}
u = \frac{1}{4} (t + b x + 2 c y + 3 d z) .
\end{equation}
Taking the determinant of $N_{\EE}^{(\alpha )} $ shows that the norm of $\alpha $ may be expressed as a homogeneous quartic polynomial in $u,x,y,z$. Replacing $u$ by the right hand side of \eqref{equu} in the equation $N_{\EE / \Q }(\alpha ) = 1$, we find the equation
\begin{equation}\label{quartnorm}
t^4 - 2 \mathcal{G} t^2 - 8 \mathcal{H} t  + \mathcal{F} = 256 ,
\end{equation}
where 
\begin{eqnarray*}
\mathcal{G}(x, y, z) & = & (3 b^2 - 8 a c) x^2 + (4 b c - 24 a d) x y + (4 c^2 - 8 b d - 16 a e) y^2 \\
                     &   & + (2 b d - 32 a e) x z + (4 c d - 24 b e) y z + (3 d^2 - 8 c e) z^2 ,
\end{eqnarray*}
$\mathcal{H}(x, y, z)$ is a homogeneous ternary cubic polynomial, and $\mathcal{F}(x, y, z)$ is a homogeneous ternary quartic polynomial. Let $I$ and $J$ denote the invariants
\begin{align*}
I & = 12 a e - 3 b d + c^2, & J & = 72 a c e + 9 b c d - 27 a d^2 - 27 b^2 e - 2 c^3 .
\end{align*}
The ternary forms $\mathcal{F}(x, y, z)$, $\mathcal{G}(x, y, z)$, $\mathcal{H}(x, y, z)$ satisfy the syzygy 
\begin{equation*}
g_4^3 - 48 g_4 I v^2 - 64 J v^3 = 27 g_6^2 
\end{equation*}
of classical invariant theory; see \cite{Cremona2}, where $g_4 = \mathcal{G} (x^2, x, 1)$, $g_6 = \mathcal{H} (x^2, x, 1)$, and $v = \mathcal{V} (x, 1)$. Notice that 
\begin{equation*}
\mathcal{G} \left( x^2,x y,y^2 \right) = - \frac{1}{3} \det \left(
\begin{array}{cc}
 \mathcal{V}_{xx} & \mathcal{V}_{xy} \\
 \mathcal{V}_{yx} & \mathcal{V}_{yy} \\
\end{array}
\right) .
\end{equation*}
These results are analogous to similar identities relating the norm of an algebraic integer of a cubic field to covariants of a binary cubic form in Cayley's syzygy $\mathcal{F}^2 + 27 \Delta \mathcal{C}^2 = 4 \mathcal{Q}^3$, where in this case, $\mathcal{C}(x ,y) = (a,b,c,d)$ defines the cubic field $\K = \Q(\delta )$ of discriminant $\Delta $, $\delta$ is a real root of $\mathcal{C}(x ,1)$, $\mathcal{Q}$ is the Hessian of $\mathcal{C}$, and $\mathcal{F}$ is the Jacobian of $\mathcal{C}$. Algebraic integers $u + x \rho_1 + y \rho_2$ of norm $1$ in a cubic field are in one to one correspondence with solutions to the equation $t^3 - 3 t \mathcal{Q} + \mathcal{F} = 27$.

In \cite{cfg2}, the binary cubic forms $(a,b,c,d)$ that were used to define a cubic field were often considered to be reduced. This means that we can assume that $a, b$ satisfy certain bounds, which is helpful in proving results about the implementation of Voronoi's algorithm; see \cite{DnF12,cfg2}. Likewise, binary quartic forms can be reduced, as Cremona \cite{Cremona2} has shown. This may be helpful in discussions on the geometry of numbers in which a quartic field is involved.   

\section{Existence of essential forms}\label{essent}

An irreducible binary form of degree $n$ which has discriminant equal to that of a number field of degree $n$ has a very simple formula for an integral basis of the ring of integers of the number field irrespective of whether there is a power integral basis. This was demonstrated in Theorem \ref{inbasthm}. In this section we consider the question of the existence of such a binary form of degree $n$. 

Ash, Brakenhoff, and Zarrabi \cite{Ash} considered the probability that a randomly chosen irreducible monic polynomial of degree $n$ has polynomial discriminant equal to the discriminant of the number field of the same degree generated by one of the roots of the polynomial. If such a monic polynomials exists, then they called it essential. They  estimated that the probability that a randomly chosen irreducible monic polynomial of degree $n \ (\geq 2)$ and height $\leq X$ is essential approaches $\frac{6}{\pi^2} \ (\approx 0.608)$ as $X \longrightarrow \infty$. In this section we are concerned with such polynomials that are in general not monic. When $n = 3$, we can easily prove that there are always irreducible cubic polynomials in $\Z[x,y]$ whose discriminant is the same as a given cubic field. These are called index forms. However, when $n > 3$, index forms are no longer binary forms. For this reason we will refer to a binary form $\mathcal{B}(x, y) \in \Z [x, y]$ of degree $n$ as {\em essential} if the discriminant of $\mathcal{B}(x, y)$ is equal to the discriminant of a number field of degree $n$. It seems that essential forms do not always exist for number fields of degree $n \geq 4$. However, this does not present too many problems for our formula for the arithmetic matrices since we can easily adjust this to accommodate the required modifications.

We will now assume that we have an arbitrary integral basis for the ring of integers and attempt to compute an essential form from the basis. Let $\EE = \Q (\zeta )$, and let $\mathcal{A} = \left\{ 1, \omega_1 , \omega_2 , \dots , \omega_{n-1} \right\}$ be a basis for the ring of integers $\mathcal{O}_{\EE }$. Let $\kappa_{j}$ be the embeddings of $\EE $ in $\C$, where $\kappa_{j} \left( \omega_i \right) = \omega_{j}^{(i)}$, and let 
\begin{align}\label{biggamma}
\Gamma_{\EE } & = \left(
\begin{array}{ccccc}
 1 & \omega_1 & \omega_2 & \dots & \omega_{n-1} \\
 1 & \omega_1^{(1)} & \omega_2^{(1)} & \dots & \omega_{n-1}^{(1)} \\
 1 & \omega_1^{(2)} & \omega_2^{(2)} & \dots & \omega_{n-1}^{(2)} \\
 \vdots & \vdots & \vdots & \ddots & \vdots \\
 1 & \omega_1^{(n-1)} & \omega_2^{(n-1)} & \dots & \omega_{n-1}^{(n-1)} \\
\end{array}
\right) , & \Delta & = \det \left( \Gamma_{\EE } \right)^2 .
\end{align}
Then by definition, $\Delta $ is the discriminant of $\EE$. Now consider the irreducible binary form 
\begin{equation*}
\mathcal{B}(x, y) = a_1 x^n + a_2 x^{n-1} y + \dots a_{n+1} y^n ,
\end{equation*}
where $a_1 a_{n+1} \not= 0$. The discriminant of $\mathcal{B}(x, y)$ is given by 
\begin{equation*}
D_{\mathcal{B}} = a_1^{2(n-1)} \prod_{i < j} \left( \delta_i - \delta_j \right)^2 ,
\end{equation*}
where $\delta_i$ for $i = 1, 2, \dots n$ are the distinct roots of $\mathcal{B}(x, 1)$. Janson \cite{Janson} showed that the discriminant can be computed via the determinant of the $(2 n - 1) \times (2 n - 1)$ Sylvester matrix
\begin{equation*}
S = \left(
\begin{array}{ccccccc}
 a_1 & a_2 & a_3 & \dots & a_{n+1} & 0 & 0 \\
 0 & a_1 & a_2 & \dots & a_n & a_{n+1} & 0 \\
 0 & 0 & a_1 & \dots & a_{n-1} & a_{n} & a_{n+1} \\
 \vdots & \vdots & \vdots & \vdots & \vdots & \vdots & \vdots \\
 n a_1 & (n-1) a_2 & (n-2) a_3 & \dots & 0 & 0 & 0 \\
 0 & n a_1 & (n-1) a_2 & \dots & 0 & 0 & 0 \\
 \vdots & \vdots & \vdots & \ddots & \vdots & \vdots & \vdots \\
 0 & 0 & 0 & \dots & 2 a_{n-1} & a_{n} & 0 \\
 0 & 0 & 0 & \dots & 3 a_{n-2} & 2 a_{n-1} & a_n \\
\end{array}
\right) ,
\end{equation*}
so that 
\begin{equation*}
D_{\mathcal{B}} = \frac{(-1)^{s}}{a_1} \det (S) ,
\end{equation*}
where $s = 1$ if $n \equiv 2, 3 \pmod{4}$ and $s = 0$ if $n \equiv 0, 1 \pmod{4}$.

For cubic fields, Davenport and Heilbronn \cite{DH} proved that 
\begin{equation}\label{DHid}
\frac{1}{\sqrt{\Delta }} \prod_{i < j} \left( \left( \omega_1^{(i)} - \omega_1^{(j)} \right) x + \left( \omega_2^{(i)} - \omega_2^{(j)} \right) y \right) = \mathcal{B}(x, y)
\end{equation}
is a binary cubic form of discriminant $\Delta $. Notice that 
\begin{equation*}
\Gamma_{\EE }^{-1} = \frac{1}{\sqrt{\Delta }} \left(
\begin{array}{ccc}
 \omega_{1}^{(1)} \omega_{2}^{(2)} - \omega_{2}^{(1)} \omega_{1}^{(2)} & \omega_{1}^{(2)} \omega_{2} - \omega_{2}^{(2)} \omega_{1} & \omega_{2}^{(1)} \omega_{1} - \omega_{1}^{(1)} \omega_{2} \\
 \omega_{2}^{(1)} - \omega_{2}^{(2)} & \omega_{2}^{(2)} - \omega_{2} & \omega_{2} - \omega_{2}^{(1)} \\
 \omega_{1}^{(2)} - \omega_{1}^{(1)} & \omega_{1} - \omega_{1}^{(2)} & \omega_{1}^{(1)} - \omega_{1} \\
\end{array}
\right) .
\end{equation*}
If we multiply $\Gamma_{\EE }^{-1}$ on the left by $(0, y, x)$ and take the product the entries in the result, then we get the binary cubic form $\mathcal{B}(x, y)$. Alternatively, where $e_j$ are the basis vectors for $\R^n$, we have
\begin{equation*}
\mathcal{B}(x, y) = \det \left( \sum_{j = 1}^{n} e_j (0, y, x) \Gamma_{\EE }^{-1} e_j e_j^T \right) .
\end{equation*} 
To prove the claim about the discriminants in the identity \eqref{DHid}, we let 
\begin{align*}
p_{21} & = \omega_{2}^{(1)} - \omega_{2}^{(2)} , & p_{22} & = \omega_{2}^{(2)} - \omega_{2} , & p_{23} & = \omega_{2} - \omega_{2}^{(1)} , \\
p_{31} & = \omega_{1}^{(2)} - \omega_{1}^{(1)} , & p_{32} & = \omega_{1} - \omega_{1}^{(2)} , & p_{33} & = \omega_{1}^{(1)} - \omega_{1} .
\end{align*}
Then 
\begin{eqnarray*}
     \sqrt{\Delta } \mathcal{B}(x, y) & = & \prod_{i < j} \left( \left( \omega_1^{(i)} - \omega_1^{(j)} \right) x + \left( \omega_2^{(i)} - \omega_2^{(j)} \right) y \right) , \\
                              & = & \left( - p_{33} x + p_{23} y \right) \left( p_{32} x - p_{22} y \right) \left( - p_{31} x + p_{21} y \right) , \\
                              & = & A_1 x^3 + A_2 x^2 y + A_3 x y^2 + A_4 y^3 ,
\end{eqnarray*}
where 
\begin{align*}
A_1 & =  p_{31} p_{32} p_{33} , & A_2 & = - p_{23} p_{31} p_{32} - p_{21} p_{33} p_{32} - p_{22} p_{31} p_{33} , \\
A_4 & = - p_{21} p_{22} p_{23} , & A_3 & = p_{22} p_{23} p_{31}+p_{21} p_{23} p_{32}+p_{21} p_{22} p_{33} .
\end{align*}
Computing discriminants, 
\small 
\begin{eqnarray*}
\frac{(-1)^{s} \det \left( \Gamma_{\EE } \right)^4 }{a_1} \det (S) & = & \frac{- 1}{A_1} \det \left(
\begin{array}{ccccc}
 A_1 & A_2 & A_3 & A_4 & 0 \\
 0 & A_1 & A_2 & A_3 & A_4 \\
 3 A_1 & 2 A_2 & A_3 & 0 & 0 \\
 0 & 3 A_1 & 2 A_2 & A_3 & 0 \\
 0 & 0 & 3 A_1 & 2 A_2 & A_3 \\
\end{array}
\right) , \\
        & = & \left( p_{22} p_{31} - p_{21} p_{32} \right)^2 \left( p_{23} p_{31} - p_{21} p_{33}\right)^2 \left( p_{23} p_{32} - p_{22} p_{33}\right)^2 . 
\end{eqnarray*}
\normalsize 
Let $P = \det \left( \Gamma_{\EE } \right) \Gamma_{\EE }^{-1} $. Then $\det (P) P^{-1} = \det \left( \Gamma_{\EE } \right)^{n-2} \Gamma_{\EE } $. Matching the left columns shows that  
\begin{equation*}
p_{22} p_{33}-p_{23} p_{32} = p_{23} p_{31}-p_{21} p_{33} = p_{21} p_{32}-p_{22} p_{31} = \det \left( \Gamma_{\EE } \right) . 
\end{equation*}
It follows that 
\begin{equation*}
\frac{(-1)^{s} \det \left( \Gamma_{\EE } \right)^4 }{a_1} \det (S) = \det \left( \Gamma_{\EE } \right)^6 ,
\end{equation*}
and dividing by $\det \left( \Gamma_{\EE } \right)^4$ proves that there is a binary cubic form with discriminant equal to that of a cubic field. This suggests a similar approach for number fields of degree $n$. However, it is not immediately clear how to generalize this proof to fields of higher degree. The difficulty in generalizing the proof of the Davenport-Heilbronn correspondence indicates that some numerical calculations may illuminate the question. When $n = 4$, and we define $\mathcal{B}_{ij}(x, y) = \frac{1}{\Delta } \prod_{k = 1}^{4} \left( p_{ik} x - p_{j k} y \right)$, where $P = \det \left( \Gamma_{\EE } \right) \Gamma_{\EE}^{-1} = \left[ p_{ij} \right]$, we can show that the discriminant of $\mathcal{B}_{i j}(x, y)$ is equal to $\prod_{\ell < m} \left( \omega_{q - 1}^{(m-1)} - \omega_{q - 1}^{(\ell -1)} \right)^2$, the discriminant of an element $\omega_q$ of the set of generators of the given integral basis for $\mathcal{O}_{\EE }$, where $\{ i, j \} \cup \{ 1, q \} = \{ 1,2,3,4 \}$ and $\{ i, j \} \cap \{ 1, q \} = \emptyset$.

We will call the pair $\left[ a_0 , \ \mathcal{B}(x, y) \right]$ an {\em essential pair} if the discriminant of the irreducible binary form $\mathcal{B}$ of degree $n$ is equal to $\Delta a_0^2$ for some rational integer $a_0$ such that $a_0^2$ divides $a_1$ and $a_0$ divides $a_2$, where $\Delta $ is the discriminant of a number field of degree $n$. In Table \ref{quarticindf} we display essential pairs for quartic fields of discriminant $\Delta $. The table indicates that perhaps essential forms $\left( a_0 = 1 \right)$ for quartic fields exist whenever there is a power integral basis for $\mathcal{O}_{\EE}$. The exhibited forms $\mathcal{B}(x, y)$ in Table \ref{quarticindf} have not been reduced. 

We know with certainty that essential forms exist for number fields of degree less than $4$. Next we will generalize Theorem \ref{inbasthm} to accommodate the possibility that there may be no essential form for a given number field of discriminant $\Delta$ and degree $n$. 

\begin{theorem}\label{newbasthm}
Let 
\begin{equation}\label{eqfpolt}
f(x) = a_1 x^n + a_2 x^{n-1} + \dots a_{n} x + a_{n+1} 
\end{equation}
be an irreducible polynomial of degree $n$ in $\Z [x]$ such that $\left[ a_0 , \ \mathcal{B}(x, y) \right]$ is an essential pair, where $f(x) = \mathcal{B}(x, 1)$, the discriminant of $\mathcal{B}$ is equal to $\Delta a_0^2$, $\Delta $ is the discriminant of a number field $\EE = \Q (\zeta )$ of degree $n$ over $\Q $, where $f(\zeta ) = 0$. Then an integral basis for the ring of integers of $\EE$ is given by $\left\{ \omega_0 , \omega_1 , \omega_2 , \dots , \omega_{n-1} \right\}$, where 
\begin{align}\label{ibsn}
 \omega_0 & = 1, & \omega_1 & = \frac{a_1}{a_0} \zeta , & \omega_j & = \sum_{k = 1}^{j} a_k \zeta^{j+1-k} \ (j > 1) .
\end{align}
\end{theorem}

\begin{proof}
We let $\zeta_i $ be the roots of $f(x)$, where $\zeta_0 = \zeta $, 
\begin{align}\label{wolk}
\Xi & = \left(
\begin{array}{ccccc}
 1 & \zeta &   \dots & \zeta^{n-1} \\
 1 & \zeta_{1}  & \dots & \zeta_{1}^{n-1} \\
 1 & \zeta_{2}  & \dots & \zeta_{2}^{n-1} \\
 1 & \zeta_{3}  & \dots & \zeta_{3}^{n-1} \\
 \vdots & \vdots  & \ddots & \vdots \\
 1 & \zeta_{n-1} & \dots &  \zeta_{n-1}^{n-1} \\
\end{array}
\right) , & \overline{A} & = \left(
\begin{array}{cccccc}
 1 & 0 & 0 & 0 & \dots & 0 \\
 0 & a_1 & a_2 & a_3 & \dots & a_{n-1} \\
 0 & 0 & a_1 & a_2 & \dots & a_{n-2} \\
 0 & 0 & 0 & a_1 & \dots & a_{n-3} \\
 \vdots & \vdots & \vdots & \vdots & \ddots & \vdots \\
 0 & 0 & 0 & 0 & \dots & a_1 \\
\end{array}
\right) Z ,
\end{align}
where $Z$ is the $n \times n$ diagonal matrix with diagonal entries equal to $1$, except possibly when $i = j = 2$, where the entry in the second row and column is $\frac{1}{a_0}$. Let $\Gamma_{\EE } = \left[ \kappa_{i-1} \left( \omega_{j-1} \right) \right]$, where $\kappa_{i-1}$ for $i = 1$ to $n$ are the embeddings of $\EE $ in $\C$, then we have $\Gamma_{\EE } = \Xi \overline{A} $. Since $a_0^n \mid a_1^{n-1}$, we have $a_0 \mid a_1$ so the entries of $\overline{A}$ are rational integers. Taking the square of the determinants of these matrices, 
\begin{equation}\label{discnew}
 \det \left( \Gamma_{\EE} \right)^2 = \frac{a_{1}^{2 n-2}}{a_0^2} \prod_{i < j} \left( \zeta_{i} - \zeta_{j} \right)^2 = \Delta . 
\end{equation}
The rest of the proof is almost the same as that of Theorem \ref{inbasthm}. However, we must show that $\omega_1 \ \left( = \frac{a_1}{a_0} \zeta \right)$ is an algebraic integer. Multiplying \eqref{eqfpolt} by the rational integer $\frac{a_1^{n-1}}{a_0^n}$, we see that $\omega_1$ is a root of the monic polynomial 
\begin{equation*}
X^n + \frac{a_2}{a_0} X^{n-1} + \frac{a_1 a_3}{a_0^2} X^{n-2} + \frac{a_1^2 a_4}{a_0^3} X^{n-3} + \dots + \frac{a_1^{n-2} a_{n}}{a_0^{n-1}} X + \frac{a_1^{n-1} a_{n+1}}{a_0^{n}} ,
\end{equation*}
with coefficients in $\Z $. 

Next, let $Z$ be the diagonal matrix on the right of \eqref{wolk}. The arithmetic matrix $\overline{N_{\EE}^{(\alpha )}}$ for the basis given in this theorem is obtained by replacing $x_1$ with $\frac{x_1}{a_0}$ in $Z^{-1} M_{\EE}^{(\alpha )} Z$, where $M_{\EE}^{(\alpha )}$ is given by \eqref{defaijone} to \eqref{defaijfive}. We obtain this expression since $$\alpha = x_0 + x_1 \omega_1 + \dots + x_{n-1} \omega_{n-1} = x_0 + z_1 \rho_1 + x_2 \rho_2 + \dots + x_{n-1} \rho_{n-1} ,$$ where $z_1 = \frac{x_1}{a_0}$ and $\rho_j = \omega_j$ for $j \not= 1$, $\rho_1 = a_1 \zeta $.

Let $T_{\EE }$ be the matrix obtained by replacing $x_1$ with $\frac{x_1}{a_0}$ in \eqref{defaijone} to \eqref{defaijfive}. Assuming $n > 2$, the entries in the first three rows of 
\begin{equation}\label{newarithm}
\overline{N_{\EE}^{(\alpha )}} = Z^{-1} T_{\EE } Z
\end{equation}
for $2 \leq j \leq n$ are given by
\begin{align*}
& (j = 2)  & & (j = 3, \dots, n-1) \\
 a_{1 \ 2} & = - \frac{a_1 a_{n+1}}{a_0} x_{n-1} , & a_{1 \ j} & = - a_{n+1} \sum_{k = 1}^{j-1} a_k x_{k+n-j} ,  \\
 a_{2 \ 2} & = x_0 - \frac{a_2}{a_0} x_1 - \sum_{k = 3}^{n+1} a_k x_{k-1} , & a_{2 \ j} & = - a_j x_1 - a_0 \sum_{k = j+1}^{n+1} a_k x_{k+1-j} , \\
 a_{3 \ 2} & = \frac{a_1}{a_0^2} x_1 , & a_{3 \ j} & = \delta_{3 \ j} \ x_0 - \sum_{k = j}^{m_{3 \ j}} a_k x_{k+2-j} , \\
\end{align*}
and
\begin{align*}
& (j = n) & \\
a_{1 \ n} & = - \frac{a_1 a_{n+1}}{a_0} x_1 - a_{n+1} \sum_{k = 2}^{n-1} a_k x_{k} , \\
a_{2 \ n} & = - a_n x_1 - a_0 a_{n+1} x_2 , \\
a_{3 \ n} & = \delta_{3 \ n} \ x_0 - \sum_{k = n}^{m_{3 \ n}} a_k x_{k+2-n} .
\end{align*}
The remaining entries of $\overline{N_{\EE}^{(\alpha )}} = Z^{-1} T_{\EE } Z$ have the following expressions:
\begin{align*}
 a_{i \ 1} & = x_{i-1} , \ \text{for} \ i \geq 1 , & a_{i \ j} & = \sum_{k = 1}^{j-1} a_k x_{k+i-j-1} \ \text{for} \ i - 2 \geq j \geq 3 , \\
 a_{i \ 2} & = \frac{a_1}{a_0} x_{i-2} \ \text{for} \ i \geq 4 , & a_{i \ i-1} & = \frac{a_1}{a_0} x_1 + \sum_{k=2}^{i-2} a_k x_k , \ \text{for} \ i \geq 4 , \\
 m_{i \ j} & = \min (n - i + j, n + 1) ,  & a_{i \ j} & = \delta_{ij} x_0 - \sum_{k = j}^{m_{ij}} a_k x_{k+i-j-1} \ \text{for} \ j \geq i \geq 3 .
\end{align*} 
Since $a_0^2 \mid a_1$, the entries of the matrix $\overline{N_{\EE}^{(\alpha )}}$ belong to $\Z \left[ x_0, x_1, \dots, x_{n-1} \right]$. All of the $\omega_j$ are algebraic integers just as in the proof of Theorem \ref{inbasthm}. 
\end{proof}

From this point forward, we will use $N_{\EE}^{(\alpha )}$ to denote an arithmetic matrix determined by an integral basis for $\mathcal{O}_{\EE }$. Most often the entries of $N_{\EE}^{(\alpha )}$ will be given by the equations that immediately follow \eqref{newarithm}. We are now in a position to state the main result of this article. The proof is almost identical to that of Proposition \ref{mainres}. 

\begin{proposition}\label{genclaim}
In terms of the integral basis of the ring of integers given Theorem \ref{newbasthm}, the analogue of Proposition \ref{mainres} holds for the arithmetic matrix given by \eqref{newarithm}. 
\end{proposition}

\begin{example}
Let $\EE = \Q (\zeta )$ be a quartic field of discriminant $\Delta = 513$, where $\zeta $ is a root of $\mathcal{B}(x, 1)$, $\mathcal{B} = (a,b,c,d,e) = (4, - 2, - 3, 1, 1)$. The discriminant of $\mathcal{B}$ is $4 \Delta$. It appears that there is may not exist an essential binary quartic form for this field. Nevertheless, there is an integral basis for the ring of integers $\mathcal{O}_{\EE}$ given by $\left\{ 1, \omega_1, \omega_2, \omega_3 \right\}$, where $\omega_1 = \frac{1}{2} a \zeta $, $\omega_2 = a \zeta^2 + b \zeta $, $\omega_3 = a \zeta^3 + b \zeta^2 + c \zeta $. An arithmetic matrix can still be determined, and easily found using a precise numerical approximation of $\Gamma_{\EE}$. However, the resulting formulas for the entries of this arithmetic matrix will not be the same as those given in \eqref{defaijone} to \eqref{defaijfive}, so that the following matrix is not correct for this number field because it does not facilitate arithmetic in the maximal order $\mathcal{O}_{\EE}$: 
\begin{equation*}
 \left(
\begin{array}{cccc}
 u & -4 z & 2 z-4 y & -4 x+2 y+3 z \\
 x & u+2 x+3 y-z & 3 x-y-z & -x-y \\
 y & 4 x & u+3 y-z & -y-z \\
 z & 4 y & 4 x-2 y & u-z \\
\end{array}
\right) .
\end{equation*}
We must make some minor adjustments as indicated in the proof of Theorem \ref{newbasthm}. The correct arithmetic matrix for this field is given by 
\begin{equation*}
N_{\EE }^{(\alpha )} = \left(
\begin{array}{cccc}
 u & -2 z & 2 z-4 y & -2 x+2 y+3 z \\
 x & u+x+3 y-z & 3 x-2 y-2 z & -x-2 y \\
 y & x & u+3 y-z & -y-z \\
 z & 2 y & 2 x-2 y & u-z \\
\end{array}
\right) .
\end{equation*}
\end{example}

\begin{remark}
Let $\EE$ be a number field of degree $n$ over $\Q$ and discriminant $\Delta $. If there exists an algebraic integer $\alpha \in \mathcal{O}_{\EE }$ of degree $n$ such that the discriminant $D (\alpha )$ of $\alpha $ divides both $\Delta N(\alpha )$ and $\Delta N(\alpha )^2 Tr \left( \alpha ^{-1} \right)^2$, then $\left[ a_0 , \mathcal{B} \right]$ is an essential pair for $\EE$, where $a_0 = \sqrt{\frac{D(\alpha )}{\Delta }}$, $\mathcal{B}(x, y) = x^n f (y/x)$, and $f(x)$ is the minimum polynomial of $\alpha $. 
\end{remark}

The existence of an essential pair for a number field of discriminant $\Delta$ and degree $n \geq 4$ seems likely. It is sensible to investigate whether for each quintic field there is a quintic essential form. We list essential pairs for some quintic fields in Table \ref{quinticindf}. It is easy to see that if $\mathcal{B}$ is an essential form for $\EE$, then $[1, \ \mathcal{B} ]$ is an essential pair for $\EE$.

\section{Matrix multiplication}\label{matmult}

In the introduction we mentioned that these results may enable efficient implementation of products of algebraic integers. Before we discuss this, it important to clarify that we are not advocating the use of an efficient matrix multiplication algorithm in order to take the product of only two algebraic integers. Instead, when we are interested in computing several products of algebraic integers, Algorithm \ref{algww} below may provide a reasonable alternative. When we wish to take a single product $\alpha \beta $, where $\alpha , \beta \in \mathcal{O}_{\EE }$ and $\EE$ has degree $n$ over $\Q$, this can be done with the product of a square arithmetic matrix and a column vector. In the most naive manner, this requires $n^2$ multiplications, although there may be ways to reduce the number of multiplications required. 

An algorithm due to Winograd \cite{Winograd2} and Waksman \cite{Waksman2} for computing $A B$, where $A$ is an $m \times m$ matrix with $m$ even and $B$ is an $m \times m$ matrix, is given below. 
\begin{algorithm}[Winograd-Waksman]\label{algww}   
{\bf Input:} Two $m \times m$ matrices $A = \left[ a_{i,j} \right]$ and $B = \left[ b_{i,j} \right]$ with $m$ even. {\bf Output:} The product $A B = \left[ c_{i,j} \right]$. 
\begin{enumerate}
\item For each $i, j$: $1 \leq i, j \leq n$, calculate 
\begin{equation*}
X_{i,j} = \sum_{k=1}^{m/2} \left( a_{i, 2 k - 1} + b_{2 k, j} \right) \left( a_{i, 2 k} + b_{2 k - 1, j} \right) .
\end{equation*}
($\frac{n^3}{2}$ multiplications)
\item For each $i$: $1 \leq i \leq m$ and $j = 1, i$, calculate 
\begin{equation*}
Z_{i,j} = \sum_{k=1}^{m/2} \left( a_{i, 2 k - 1} - b_{2 k, j} \right) \left( a_{i, 2 k} - b_{2 k - 1, j} \right) .
\end{equation*}
($\frac{n}{2} \left( m + m - 1 \right) = n^2 - \frac{m}{2}$ multiplications since we need only compute $Z_{1,1}$ once)
\item For each $i, j$: $1 \leq i \leq n$ and $j = 1, i$, calculate the sums 
\begin{equation*}
Y_{i,j} = \frac{1}{2} \left( X_{i, j} + Z_{i, j} \right) .
\end{equation*}
\item For each $i$: $2 \leq i \leq n$, calculate the sums 
\begin{equation*}
Y_{i, 1} = Y_{1,1} + Y_{i,i} - Y_{1,i} . 
\end{equation*}
\item For each $i, j$: $2 \leq i, j \leq m$, calculate the sums 
\begin{equation*}
Y_{i,j} = Y_{1,1} + Y_{i,i} + Y_{j, j} - Y_{j, i} - Y_{1, i} .
\end{equation*}
\item For each $i, j$: $1 \leq i, j \leq n$, calculate $c_{i, j} = X_{i, j} - Y_{i, j}$.
\end{enumerate}
\end{algorithm}
The total number of multiplications required to compute $A B$ is $\frac{m^3}{2} + m^2 - \frac{m}{2}$. Doing so with a recursive version of the algorithm means that $A B$ can be computed in $O \left( m^{\log_2(7)} \right)$ multiplications. If we are multiplying $\alpha \in \mathcal{O}_{\EE}$ by each of the $n$ elements $\beta_1, \beta_2, \dots , \beta_n \in \mathcal{O}_{\EE}$, then the Winograd-Waksman algorithm says that we can do this in $O \left( n^{\log_2(7)} \right)$ multiplications; via the product $N_{\EE}^{(\alpha )} U$, where the columns of the $n \times n$ matrix $U$ are respectively the coefficients of the integral basis elements $\rho_i$ in the expressions for $\beta_j$. To implement the algorithm it is necessary to pad the matrices with rows and columns of zeros so that they always have $n$ even. However, it is not necessary for $n$ to be a power of $2$; padding is done at every iteration. 

Alternatively, if we were to use the Radix-2 FFT to compute the product of two algebraic integers $\alpha , \beta \in \mathcal{O}_{\EE }$, where $\EE = \Q (\zeta )$ has degree $n$ over $\Q$, we would need to express $\alpha $ and $\beta $ in terms of powers of $\zeta $ using $\Gamma_{\EE } = \Xi \overline{A}$ and $\Xi \overline{A} \left( x_0, x_1, \dots , x_{n-1} \right)^T = \left( \alpha, \alpha^{(1)}, \dots , \alpha^{(n-1)} \right)^T$, where $\Xi $ and $\overline{A}$ are given by \eqref{wolk}. This means we must multiply two polynomials of $p(x)$ and $q(x)$ of maximum degree $n$ having coefficients in $\Q$. However, we can express $p(x)$ with coefficients that have a common denominator, so we may assume that the task is to multiply two polynomials $f(x), g(x) \in \Z [x]$ of maximum degree $n$.

\begin{algorithm}[FFT polynomial multiplication]\label{algR2FFT}   
{\bf Input:} Two polynomials $f(x)$ and $g(x)$ of maximum degree $n$. {\bf Output:} The product $f(x) g(x)$. 
\begin{enumerate}
\item Construct arrays 
\begin{align*}
F & = \left\{ f_0, f_1, \dots , f_{n}, 0, 0, \dots , 0 \right\}, & G & = \left\{ g_0, g_1, \dots , g_{n}, 0, 0, \dots , 0 \right\} 
\end{align*}
from the coefficients of the polynomials 
\begin{align*}
f(x) & = f_0 + f_1 x + \dots + f_n x^n, & g(x) & = g_0 + g_1 x + \dots + g_n x^n , 
\end{align*}
padding the ends with zeros until $F$ and $G$ have length $2 n$. 
\item Calculate the FFT of both $F$ and $G$, denoting these by $\overline{F}$ and $\overline{G}$.
\item Take the Hadamard product of $\overline{F}$ and $\overline{G}$, and denote this by $\overline{H}$. 
\item Compute the inverse FFT of $\overline{H}$, and assume this is $H = \left\{ h_0, h_1, \dots , h_{2 n} \right\}$. 
\item Return the product $f(x) g(x) = h_0 + h_1 x + \dots + h_{2 n} x^{2 n}$. 
\end{enumerate}
\end{algorithm}

Once we have computed the product of $\alpha , \beta \in \mathcal{O}_{\EE }$ using Algorithm \ref{algR2FFT}, we must next use the minimum polynomial of $\zeta $ to express the result as a polynomial in $\zeta $ with maximum degree $n$, and finally use $\Gamma_{\EE } = \Xi \overline{A}$ to return an expression in terms of the original integral basis for $\mathcal{O}_{\EE }$. The number of multiplications required to perform Algorithm \ref{algR2FFT} using the Radix-2 FFT is $O \left( n \log_2 (n) \right)$ multiplications. Of course if $n$ is large enough, this is an improvement on $O \left( n^2 \right)$ multiplications. However, there are two matrix-vector multiplications to perform in order to convert between the integral basis for $\mathcal{O}_{\EE }$ and the basis of $\EE $ so the number of multiplications required to use the FFT in multiplication of two algebraic integers is $O \left( n^2 \right)$. This analysis ignores the constants involved in counting operations, the coefficient of $n^2$. Furthermore, we must compute the FFT three times, and it is well known that the polynomials involved must be quite large in order for use of the FFT to be the optimal choice due to these constants and other practical considerations on implementing the algorithm.   

To multiply $\alpha \in \mathcal{O}_{\EE }$ by each of $\beta_1, \beta_2, \dots , \beta_n \in \mathcal{O}_{\EE }$, we let 
\begin{align*}
\alpha & = x_0 + x_1 \rho_1 + \dots + x_{n-1} \rho_{n-1} , & \beta_k & = y_0^{(k)} + y_1^{(k)} \rho_1 + \dots + y_{n-1}^{(k)} \rho_{n-1} , 
\end{align*}
where $\left\{ \rho_j \right\}$ is the integral basis for $\mathcal{O}_{\EE }$ given by Theorem \ref{newbasthm}, and compute the matrix product $\overline{N_{\EE }^{(\alpha )}} U $, where 
\begin{equation*}
U = \left(
\begin{array}{cccc}
 y_0^{(0)} & y_0^{(1)} & \dots & y_0^{(n)} \\
 y_1^{(0)} & y_1^{(1)} & \dots & y_1^{(n)} \\
 \vdots & \vdots & \ddots & \vdots \\
 y_{n-1}^{(0)} & y_{n-1}^{(1)} & \dots & y_{n-1}^{(n)} \\
\end{array}
\right) .
\end{equation*}
The matrix $N_{\EE }^{(\alpha )} U $ contains the coefficients of the $\rho_j$ in each of $\alpha \beta_k$. Using Algorithm \ref{algww} to do this would require $O \left( n^{\log_2(7)} \right)$ multiplications. It would take $O \left( n^{2 n} \right)$ multiplications to do the same with the FFT. 

\small
\begin{center}
\begin{longtable}{|c|c|c|c|c|}
\caption{Essential pairs $\left[ a_0, \left( a_1, a_2, a_3, a_4, a_5 \right) \right]$ for quartic fields of discriminant $\Delta $} \label{quarticindf} \\

\hline \multicolumn{1}{|c|}{$\Delta$} & \multicolumn{1}{c|}{$\mathcal{B}(x,y)$} & \hspace{2.0 cm} & \multicolumn{1}{|c|}{$\Delta$} & \multicolumn{1}{c|}{$\mathcal{B}(x,y)$}  \\ \hline
\endfirsthead

\multicolumn{5}{c}{\tablename\ \thetable{} -- continued} \\
\hline \multicolumn{1}{|c|}{$\Delta$} & \multicolumn{1}{c|}{$\mathcal{B}(x,y)$} & \hspace{2.0 cm} & \multicolumn{1}{|c|}{$\Delta$} & \multicolumn{1}{c|}{$\mathcal{B}(x,y)$} \\ \hline
\endhead

\hline \multicolumn{5}{|c|}{{Continued}} \\ \hline
\endfoot

\hline 
\endlastfoot

-275 & [1, \ (1, 1, 0, -2, -1)]  & & 117 & [1, \ (1, -1, -1, 1, 1)]  \\   
-283 & [1, \ (1, 1, 0, 0, -1)]  & & 125 & [1, \ (1, 1, 1, 1, 1)] \\   
-331 & [1, \ (1, -1, -1, 1, -1)]  & & 144 & [1, \ (1, 0, -1, 0, 1)]  \\
-400 & [1, \ (1, 0, 1, 0, -1)]  & & 189 & [1, \ (1, 1, 0, -2, 1)]  \\
-448 & [1, \ (1, 2, 1, -2, -1)]  & & 225 & [2, \ (4, -6, 5, -3, 1)]    \\
-475 & [1, \ (1, 1, -2, 2, -1)]  & & 229 & [1, \ (1, 0, 0, 1, 1)]  \\
-491 & [1, \ (1, 2, 2, -1, -1)]  & & 256 & [1, \ (1, 0, 0, 0, 1)]  \\
-507 & [1, \ (1, 1, -1, 1, 1)]  & & 257 & [1, \ (1, 0, 1, -1, 1)]  \\
-563 & [1, \ (1, 1, 1, 1, -1)]  & & 272 & [1, \ (1, 0, 1, 2, 1)]  \\
-643 & [1, \ (1, 1, 0, 2, 1)]  & & 320 & [1, \ (1, 2, 0, 0, 2)]  \\
-688 & [1, \ (1, 0, 0, -2, -1)]  & & 392 & [1, \ (1, 1, 0, -1, 1)]  \\
-731 & [1, \ (1, 0, -2, -1, -1)]  & & 400 & [1, \ (1, 0, 3, 0, 1)] \\
-751 & [1, \ (1, 2, 1, 1, -1)]  & & 432 & [1, \ (1, 0, 3, 0, 3)]  \\
-775 & [2, \ (4, 2, -3, -3, -1)]   & & 441 & [2, \ (4, 2, -1, 1, 1)]  \\
-848 & [1, \ (1, 0, -1, -2, 1)]  & & 512 & [1, \ (1, 0, 2, 0, 2)]  \\
-976 & [1, \ (1, -2, 3, 0, -1)]  & & 513 & [2, \ (4, - 2, - 3, 1, 1)]   \\
-1024 & [1, \ (1, 0, 2, 0, -1)]  & & 549 & [1, \ (1, -2, -2, 3, 3)]  \\
-1099 & [1, \ (1, 1, 1, 3, 1)]  & & 576 & [2, \ (4, 4, 2, 2, 1)]   \\
-1107 & [1, \ (1, 1, 0, 2, -1)]  & & 592 & [1, \ (1, -2, 4, -2, 2)]  \\
-1156 & [1, \ (1, 1, -2, 1, 1)]  & & 605 & [1, \ (1, 1, 1, -1, 1)]  \\
-1192 & [1, \ (1, -1, -2, 1, -1)]  & & 656 & [1, \ (1, 2, -1, -2, 2)]  \\
-1255 & (1, \ (1 , 0, -1, 3, -1)]  & & 657 & [2, \ (4, 2, 5, 1, 1)]   \\
-1323 & [1, \ (1, -1, -3, -1, 1)]  & & 697 & [1, \ (1, 1, 2, 1, 2)]  \\
-1328 & [1, \ (1, -2, -3, 0, 1)]  & & 725 & [1, \ (1, -1, -3, 1, 1)]  \\
-1371 & [1, \ (1, 0, 2, 1, -1)]  & & 761 & [1, \ (1, -2, 1, 1, 1)]  \\
-1375 & \ [4, \ (16, -20 , -5, 15, -5)] \ & & 784 & [2, \ (4, 0, -3, 0, 1)]   \\
-1399 & [1, \ (1, -1, 0, 1, -2)]  & & 788 & [1, \ (1, -1, 2, -2, 2)]  \\
-1423 & [1, \ (1, -1, 1, -2, -1)]  & & 832 & [1, \ (1, 2, 0, -4, 2)]  \\
-1424 & [1, \ (1, 0, 1, 2, -1)]  & & 837 & [1, \ (1, 3, 0, -6, 3)]  \\
-1456 & [1, \ (1, 0, -2, 2, 1)]  & & 873 & [2, \ (4, 2, 5, -5, 1)]   \\
-1472 & [1, \ (1, 2, 2, 0, -2)]  & & 892 & [1, \ (1, 1, 2, -3, 1)]  \\
-1472 & [1, \ (1, 2, 3, 2, -1)] & & 981 & [1, \ (1, -5, 5, 5, 1)]  \\
-1475 & [3, \ (9, 0, -8, -5, -1)] & & 985 & [1, \ (1, 1, 3, 2, 3)]  \\
-1588 & [1, \ (1, 1, -3, 0, 2)]  & & 1008 & \ [1, \ (1, 4, 1, -6, 3)] \ \\
-1600 & [2, \ (4, 4, -2, -4, -1)] & & 1008 & [1, \ (1, -4, 11, -14, 13)]  \\
-1728 & [1, \ (1, 2, 0, -4, -2)] & & 1016 & [1, \ (1, -1, 1, -2, 2)]  \\
-1732 & [1, \ (1, 3, 0, -1, -1)] & & 1025 & [2, \ (4, 2, 3, 1, 1)]   \\
-1775 & [4, \ (16, 4, -15, 7, -1)] & & 1040 & [1, \ (1, 4, 3, -4, 1)]  \\
-1791 & [1, \ (1, 5, 8, 7, 4)]  & & 1040 & \ [4, \ (16, 32 , 25 , 8 , 1)] \ \\
-1792 & [1, \ (1, 4, 4, 4, 1)]  & & 1076 & [1, \ (1, 3, 6, 4, 2)] \\
-1823 & [1, \ (1, 3, 3, 4, 1)]  & & 1088 & [1, \ (1, 2, 1, -2, 1)]  \\
-1856 & [1, \ (1, 2, 1, 2, -1)]  & & 1088 & [1, \ (1, 2, 5, 4, 2)]  \\
-1879 & [2, \ (4, -2, -5, 5, -1)]   & & 1125 & [1, \ (1, 7, 14, 8, 1)]  \\
-1927 & [1, \ (1, 7, 13, 2, -1)]  & & 1129 & [1, \ (1, 1, 0, 1, 2)]  \\
-1931 & [1, \ (1, 0, 0, 3, 1)]  & & 1161 & [4, \ (16, 20, 15, 5, 1)]   \\
-1963 & [1, \ (1, 1, 2, 2, -1)]  & & 1168 & [3, \ (9, -6, 5, -4, 1)]   \\
-1968 & [1, \ (1, 8, 20, 14, -2)] & & 1197 & [1, (1, -3, 8, 6, 1)]  \\
-1975 & \ [10, \ (100, 90, 19, -3, -1)] \ & & 1197 & [3, (9, -3, -5, 1, 1)]  \\
-1984 & [1, \ (1, -2, 1, 0, -2)]  & & 1225 & \ [6, (36, 42, 23, 7, 1)] \ \\
-1984 & [1, \ (1, -2, 2, 2, -1)]  & & 1229 & [1, (1, 1, 3, 1, 3)]  
\end{longtable}
\end{center}
\normalsize 

\small
\begin{center}
\begin{longtable}{|c|c|c|c|c|}
\caption{Essential pairs $\left[ a_0, \left( a_1, a_2, a_3, a_4, a_5 \right) \right]$ for quintic fields of discriminant $\Delta $} \label{quinticindf} \\

\hline \multicolumn{1}{|c|}{$\Delta$} & \multicolumn{1}{c|}{$\mathcal{B}(x,y)$} & \hspace{2.0 cm} & \multicolumn{1}{|c|}{$\Delta$} & \multicolumn{1}{c|}{$\mathcal{B}(x,y)$}  \\ \hline
\endfirsthead

\multicolumn{5}{c}{\tablename\ \thetable{} -- continued} \\
\hline \multicolumn{1}{|c|}{$\Delta$} & \multicolumn{1}{c|}{$\mathcal{B}(x,y)$} & \hspace{2.0 cm} & \multicolumn{1}{|c|}{$\Delta$} & \multicolumn{1}{c|}{$\mathcal{B}(x,y)$} \\ \hline
\endhead

\hline \multicolumn{5}{|c|}{{Continued}} \\ \hline
\endfoot

\hline 
\endlastfoot

-4511 & [1,\ (1, 0, 2, 1, -2, -1)]  & & 4477 & [1,\ (1, 0, 1, 0, 1, 1)]  \\   
-4930 & [1, (1, -2, -2, 3, 2, -1)]  & & 4549 & [1,\ (1, 0, 2, 2, 1, 1)] \\   
-5519 & [1,\ (1, 4, 5, 4, 4, 1)]  & & 4597 & [1,\ (1, 0, 1, 2, -1, 1)] \\
-5783 & [1,\ (1, -2, 1, 2, -2, -1)]  & & 4757 & [1,\ (1, 1, 2, 1, 2, 1)] \\
-7031 & [1,\ (1, 0, -1, 1, -1, -1)]  & & 4817 & [1,\ (1, -2, 1, -2, 2, -1)] \\
-7367 & [1,\ (1, 2, 0, -3, -2, 1)]  & & 4897 & [1,\ (1, 2, -2, 1, -2, 1)] \\
-7463 & [1,\ (1, -2, 1, 0, -2, 1)]  & & 5025 & [1,\ (1, -1, -1, 0, -1, -1)] \\
-8519 & [1,\ (1, -1, -1, 0, -1, 1)]  & & 5164 & [1,\ (1, -1, -1, 0, 2, 1)] \\
-8647 & [1,\ (1, 2, 0, 1, 2, -1)]  & & 5437 & [1,\ (1, 0, 2, 2, 2, 1)] \\
-9439 & [1,\ (1, 1, -1, -1, -2, -1)]  & & 5501 & [1,\ (1, -1, 2, 0, -2, 1)] \\
-9759 & [1,\ (1, 2, -1, 0, 2, -1)]  & & 5584 & [1,\ (1, 1, 0, 0, 1, -1)] \\
-10407 & [1,\ (1, -3, 0, 3, 1, 1)]  & & 5653 & [1,\ (1, 1, 0, 0, 2, 1)] \\
-11119 & [1,\ (1, 1, -2, -3, -3, -1)]  & & 5753 & [1,\ (1, -1, 2, -1, -1, 1)] \\
-11551  & [1,\ (1, 0, -2, 3, 2, -3)]  & & 5864 & [1,\ (1, -2, 2, 1, -1, 1)] \\
-12447  & [1,\ (1, 0, 1, -1, -3, -1)]  & & 5913 & [1,\ (1, 1, 2, 1, 1, -1)] \\
-13219  & [1,\ (1, 0, 0, -2, -1, 1)]  & & 6241 & [1,\ (1, 1, 1, 2, 3, 1)] \\
-13523  & [1,\ (1, -2, -3, 0, 3, -1)] & & 6449 & [1,\ (1, 2, 3, 1, 1, -1)] \\
-13799  & [1,\ (1, -3, 2, -1, -1, 1)]  & & 6581 & [1,\ (1, 1, -2, -2, 1, 2)] \\
-13883  & [1,\ (1, 1, -2, 0, 1, -2)]  & & 6757 & [1,\ (1, -1, 0, 3, -2, 1)] \\
-14103  & [1,\ (1, -2, -1, 2, -2, 1)]  & & 6793 & [1,\ (1, -1, 0, -1, -1, -1)] \\
-14631  & [1,\ (1, 0, -3, 3, 1, -1)]  & & 7096 & [1,\ (1, 0, 0, 1, 1, -1)] \\
-14891  & [1,\ (1, -3, 0, 1, -2, -1)]  & & 7177 & [1,\ (1, -2, -1, 1, 3, 1)] \\
-14911  & [1,\ (1, -3, 0, -1, -3, -1)]  & & 7265 & [1,\ (1, 1, 1, -2, -1, -1)] \\
-15536  & [1,\ (1, -3, -2, 2, 3, 1)]  & & 7333 & [1,\ (1, -2, 2, -2, 3, -1)] \\
-15919  & [1,\ (1, -4, -5, 5, 5, 1)]  & & 7373 & [1,\ (1, -2, 0, 2, 1, -1)] \\
-16816  & \ [1,\ (1, 3, 0, -4, -1, -1)] \  & & 7376 & \ [1,\ (1, 2, 2, 0, -2, -2)] \ \\
\end{longtable}
\end{center}
\normalsize 

\section*{Acknowledgments} 

I thank Keith Matthews for asking the conditions under which there exists a binary form whose discriminant is equal to that of some number field of the same degree. This lead to the inclusion of Section \ref{essent}, some additional references, and the correction of several miss-prints.


\begin{thebibliography}{99} 

\bibitem{Ash} A. Ash, J. Brakenhoff, T. Zarrabi, {\em Equality of polynomial and field discriminants}, Experimental Mathematics, {\bf 16}, (2007) no. 3. MR2367325 (2008i:11129)

\bibitem{Belabas2} K.\ Belabas, {\em A fast algorithm to compute cubic fields}, Math. Comp. {\bf 66} (1997), no. 219, 1213--1237. MR1415795 (97m:11159)

\bibitem{CT2} J.\ W.\ Cooley and J.\ W.\ Tukey, {\em An algorithm for the machine calculation of complex Fourier series}, Math. Comp., {\bf 19} (1965), no. 90, 297--301. MR0178586 (31 \#2843)

\bibitem{Cremona2} J.\ E.\ Cremona, {\em Reduction of binary cubic and quartic forms}, LMS J. Comput. Math. {\bf 2}, (1999), 64--94. MR1693411 (2000f:11040) 

\bibitem{DH} H.\ Davenport, H.\ Heilbronn, {\em On the density of discriminants of cubic fields II}, Proc.\ Royal Soc. London Ser.\ A \textbf{322} (1971), 405--420. MR0491593 (58 \#10816)

\bibitem{DnF12} B.\ N.\ Delone, D.\ K.\ Faddeev, The theory of Irrationalities of the third degree, Translations of Mathematical Monographs, Vol. 10, American Mathematical Society, Providence, R.I. 1964. MR0160744 (28 \#3955) 

\bibitem{Drazin2} M.\ P.\ Drazin, {\em Some generalizations of matrix commutativity}, Proc. London Math. Soc., {\bf 3} (1951), no. 1, 222--231. MR0043760 (13,312c)

\bibitem{Frobenius2} G.\ Frobenius, {\em {\"U}ber lineare Substitutionen und bilineare Formen}, Journal f{\"u}r Math. {\bf 84} (1878), 1--63.

\bibitem{Goodman2} R.\ Goodman, N.\ R.\ Wallach, {\em Representations and Invariants of the Classical Groups}, Cambridge University Press, 1998. MR1606831 (99b:20073)

\bibitem{hamblbrahm2} S.\ A.\ Hambleton, {\em A cubic generalization of Brahmagupta's identity}, J. Ramanujan Math. Soc., {\bf 32} (2017), no. 4, 327--337. MR3733759

\bibitem{hambpara2} --------------, {\em Arithmetic matrices for number fields II: Parametrization of rings by binary forms}, \htmladdnormallink{arXiv:1808.10214}{https://arxiv.org/abs/1808.10214}

\bibitem{HambLemm12} --------------, F.\ Lemmermeyer, {\em Arithmetic of Pell surfaces}, Acta Arith. {\bf 146} (2011), no. 1, 1--12. MR{2741187 (2012b:11097)}

\bibitem{cfg2} --------------, H.\ C.\ Williams, Cubic fields with geometry, CMS Books in Mathematics, Springer, (to appear in print). \\
\url{https://www.springer.com/us/book/9783030014025}



\bibitem{Janson} S. Janson, {Resultant and discriminant of polynomials}, \url{http://www2.math.uu.se/~svante/papers/sjN5.pdf}

\bibitem{Lang2} S.\ Lang, Linear Algebra, Undergraduate Texts in Mathematics, 3d ed., Springer-Verlag, New York, 1987. MR0874113 (88d:15001)

\bibitem{LeBruyn42} L.\ Le Bruyn, {\em Generic norm one tori}, Nieuw Arch. Wiskd. (4) {\bf 13} (1995), no. 3, 401--407. MR1378805 (97d:14070)

\bibitem{Lem0312} F.\ Lemmermeyer, {\em Conics -- A poor man's elliptic curves}, \\
\htmladdnormallink{arXiv:math/0311306v1}{http://arxiv.org/abs/math/0311306}

\bibitem{Lemnormtori42} --------------, {\em Arithmetic of number fields from a geometric point of view}, unpublished. 

\bibitem{LemmParam72} --------------, {\em Parametrization of algebraic curves from a number theorist's point of view}, Amer. Math. Monthly, {\bf 119} (2012) no. 7, 573-583. MR2956427



\bibitem{Strassen2} V.\ Strassen, {\em Gaussian elimination is not optimal}, Numer. Math. {\bf 13} (1969), 354--356. MR0248973 (40 \#2223)


\bibitem{Voskbook42} V.\ E.\ Voskresenski{\u\i}, Algebraic groups and their birational invariants, Translations of Mathematical Monographs, {\bf 179}, Translated from the Russian manuscript by Boris {\`E}. Kunyavski{\u\i}, American Mathematical Society, Providence, RI, 1998. MR1634406 (99g:20090)

\bibitem{Vosktwodimone42} --------------, Two-dimensional algebraic tori, Izv. Akad. Nauk SSSR SER. MAT. {\bf 29} (1965), 239-244; ENGLISH TRANSL., AMER. MATH. SOC. TRANSL. (2)73 (1968), 190-195. (MR 30 \# 3097).

\bibitem{Vosktwodimtwo42} --------------, On two-dimensional algebraic tori. II, (Russian) Izv. Akad. Nauk SSSR SER. MAT. {\bf 31} (1967), no. 3, 711--716. MR0214597 (35 \# 5446)

\bibitem{Waksman2} A.\ Waksman, {\em On Winograd's algorithm for inner products}, IEEE Transactions on Computers, {\bf 19} (1970), no. 4, 360--361. MR0455534 (56 \# 13772)

\bibitem{Washington2} L.\ C.\ Washington, Introduction to Cyclotomic Fields. Graduate Texts in Mathematics, 83. Springer-Verlag, New York, 1982. MR0718674 (85g:11001)

\bibitem{Weyl2} H.\ Weyl, Algebraic Theory of Numbers, Reprint of the 1940 original. Princeton Landmarks in Mathematics. Princeton Paperbacks. Princeton University Press, Princeton, NJ, 1998. MR1617068 (98m:11110) 

\bibitem{Winograd2} S.\ Winograd, {\em A new algorithm for inner product}, IEEE Transactions on Computers {\bf 17} (1968), no. 7, 693--694.

\bibitem{Williamson2} J.\ Williamson, {\em The simultaneous reduction of two matrices to triangular form}, American J. of Math. {\bf 57} (1935), 281--293.

\end{thebibliography}
\end{document}